\documentclass[review,oneside,a4paper]{amsart}     
\usepackage[foot]{amsaddr}

\usepackage{amsmath} 
\usepackage{amssymb}  
\usepackage{amsthm}

\usepackage{hyperref}
\usepackage{marginnote}
\usepackage[inline]{enumitem}
\usepackage{multicol} 
\usepackage{tikz}
\usepackage[normalem]{ulem}
\usepackage{centernot}

\usetikzlibrary{arrows,automata,positioning,backgrounds}

\tikzstyle{vecArrow} = [thick, decoration={markings,mark=at position
   1 with {\arrow[semithick]{open triangle 60}}},
   double distance=1.4pt, shorten >= 5.5pt,
   preaction = {decorate},
   postaction = {draw,line width=1.4pt, white,shorten >= 4.5pt}]
\tikzstyle{innerWhite} = [semithick, white,line width=1.4pt, shorten >= 4.5pt]

\newtheorem{theorem}{Theorem}[]
\newtheorem{lemma}[theorem]{Lemma}

\newtheorem{proposition}[theorem]{Proposition}
\newtheorem{corollary}[theorem]{Corollary}
\newtheorem{definition}[theorem]{Definition}

\newtheorem{example}[theorem]{Example}
\newtheorem{question}[theorem]{Question}

\newtheorem{rem}[theorem]{Remark}
\newenvironment{myproof}[2] {\paragraph{{\it Proof of {#1} {#2}.}}}{\hfill$\square$}

\newcommand{\calK}{\mathcal{K}}

\newcommand{\C}{\mathbb{C}}

\newcommand{\Cm}{\C_{-}}
\newcommand{\N}{\mathbb{N}}
\newcommand{\R}{\mathbb{R}}

\renewcommand{\Re}{\operatorname{Re}}
\newcommand{\ran}{\operatorname{ran}}

\newcommand{\vertiii}[1]{{\left\vert\kern-0.25ex\left\vert\kern-0.25ex\left\vert #1 
    \right\vert\kern-0.25ex\right\vert\kern-0.25ex\right\vert}}

\title[On continuity of solutions for parabolic control systems and ISS]{On continuity of solutions for parabolic control systems and input-to-state stability}
\author[B.~Jacob]{Birgit Jacob}
\address[BJ]{University of Wuppertal, School of Mathematics and Natural Sciences, Gau\ss \-str.\ 20, 42119 Wuppertal, Germany}
\email{bjacob@uni-wuppertal.de}
\author[F.L.~Schwenninger]{Felix L.~Schwenninger*}
\address[FLS]{Department of Mathematics, University of Hamburg, Bundesstr.\ 55, 20146 Hamburg, Germany}
\email{felix.schwenninger@uni-hamburg.de}
\author[H.~Zwart]{Hans Zwart}
\address[HZ]{Department of Applied Mathematics, University of Twente, P.O.\ Box 217, 7500AE Enschede, The Netherlands,  and Eindhoven University of Technology,
P.O.\ Box 513,
5600 MB Eindhoven}
\email{h.j.zwart@utwente.nl}
\date{November 1, *Corresponding author}
\begin{document}

\maketitle
\begin{abstract}
We study minimal conditions under which mild solutions of linear evolutionary control systems are continuous for arbitrary bounded input functions. This question naturally appears when working with boundary controlled, linear partial differential equations.
Here, we focus on parabolic equations which allow for operator-theoretic methods such as the holomorphic functional calculus. 
Moreover, we investigate stronger conditions than continuity leading to input-to-state stability with respect to Orlicz spaces. This also implies that the notions of input-to-state stability and integral-input-to-state stability coincide if additionally the uncontrolled equation is dissipative and the input space is finite-dimensional.
\end{abstract}


\section{Introduction}
Many evolutionary systems and  linear pde's can be modelled by abstract differential equations of the form 
\begin{equation}\label{eq1}
\Sigma(A,B):\quad\dot{x}(t)=Ax(t)+Bu(t), \quad x(0)=x_{0},
\end{equation}
where $A$ generates a strongly continuous semigroup $T(\cdot)$ on the Banach space $X$, $x_{0}\in X$ and the  control input $u:[0,t]\to U$ enters through the inhomogeneity $Bu$. Typical examples where $B$, as linear operator from $U$ to $X$, is unbounded are given by boundary control systems, see e.g.\ \cite[Chapter 11]{TucWei09}. In that case,  Eq.\ \eqref{eq1} is viewed on the extrapolation  space $X_{-1}\supset X$ a-priori and $B$ is bounded as operator from $U$ to $X_{-1}$. 
This setting immediately gives rise to the (formal) mild solution $x:[0,\infty)\to X_{-1}$,
$$x(t)=T_{-1}(t)x_{0}+\int_{0}^{t}T_{-1}(t-s)Bu(s)\,ds, \quad t>0,$$
for every input $u$ from a considered class $Z$ of $U$-valued functions, e.g., $Z(0,t;U)=L^{2}(0,t;U)$. 
Here $T_{-1}$ denotes the extension of the semigroup to the extrapolation space. This abstract lifting argument comes at a price: It is not even clear if the solution $x$ is an $X$-valued continuous function. A central point of this article is to study `minimal' conditions on $A$ and $B$ under which the solution $x$ is continuous for any $u\in Z$.
Necessary for the latter is that $x$ is $X$-valued, or equivalently that
\begin{equation}\label{cond:eq2}
\int_{0}^{t}T_{-1}(s)Bu(s)\,ds\in X\quad \forall t>0, u\in Z,
\end{equation}
provided that $Z$ is invariant under translations which will always be the case here. 
In this case, we call $B$ a {\it $Z$-admissible control operator}. 
The question now is whether $Z$-admissibility already implies that solutions $x$ are continuous. 
Admissible operators, in particularly for $Z=L^2$, have been studied intensively, e.g., 
\cite{JacoPart04,LeMerdy2003,TucWei09,Weiss89ii,Weiss89i}. 
For $Z=L^p$ with $p\in[1,\infty)$, the above question has an affirmative answer which follows rather directly. This was already shown by Weiss in \cite{Weiss89i}. Intriguingly,  the case $Z=L^\infty$ is still an open problem, see \cite[Problem 2.4]{Weiss89ii} and the discussion in \cite[Sec.~6]{JNPS16}.
 With results of the following type, we give a partial answer in the case of parabolic equations, (see Corollary \ref{cor1}, Theorem \ref{thm2}).
\begin{theorem}\label{thm11}Let $A$ generate an {\color{black}exponentially stable semigroup} $T$ on a Hilbert space $X$.
Then the following assertions are equivalent.
\begin{itemize}
\item[(i)] Any bounded, linear operator $B:\mathbb{C}\to X_{-1}$ is $L^{\infty}$-admissible.
\item[(ii)] The solutions $x$ of \eqref{eq1} are continuous $X$-valued functions, 
 for all $u\in L_{loc}^{\infty}(0,\infty;U)$, and $B\in\mathcal{L}(U,X_{-1})$ with $\dim U<\infty$.
\item[(iii)] $T$ is a {\color{black}bounded analytic} semigroup similar to a contraction semigroup.
\end{itemize}
\end{theorem}
We remark that the similarity to a contraction  in Theorem \ref{thm11}{\it (iii)} is a condition which is satisfied in many applications. It means nothing else than that $A$ is dissipative with respect to an equivalent Hilbertian norm. \\
 Identifying admissible operators is an interesting task in its own right:  by the closed graph theorem this is the same as  characterizing the operators $B\in\mathcal{L}(U,X_{-1})$ for which
$$\forall t>0\ \exists K>0:\quad\left\|\int_{0}^{t}T_{-1}(s)Bu(s)\,ds\right\| \leq K\|u\|_{Z(0,t;U)} \quad \forall u\in Z(0,t;U),$$
at least if the space $Z(0,t;U)$ is continuously embedded in $L^{1}(0,t;U)$.
In 1991,  Weiss \cite{Weiss91} posed the question whether 
$L^{2}$-admissibility, i.e.\ $Z=L^{2}$, 
 is  equivalent to  $$\sup\nolimits_{\Re\lambda>\omega_{0}}\|\sqrt{\Re\lambda}(\lambda-A)^{-1}B\|_{\mathcal{L}(U,X)}< \infty$$ for sufficiently large $\omega_{0}>0$. By setting $u(s)=e^{-s\lambda}$, the necessity of this condition is easy to see.
Counterexamples where $U$ is not a Hilbert space were already mentioned in \cite{Weiss91}. However, the question for Hilbert spaces $U$ and $X$ was the starting point of intensive research around what has become known as the \textit{Weiss conjecture}, see \cite{JacoPart04,TucWei09} for surveys. Although even in this case, counterexamples  were found \cite{JaPaPo02,JaZwa04,ZwartJacobStaffans}, there are situations with positive answers --- most prominently,  the case of contraction semigroups and $U=\mathbb{C}$, \cite{JaPa01}, in which a connection with 
deep results in complex analysis appears. 
In \cite{LeMerdy2003} Le Merdy characterized when the Weiss conjecture for bounded analytic semigroups and any space $U$ is true --- by drawing a link to the $H^{\infty}$-functional calculus. For bounded analytic semigroups on Hilbert spaces, the latter can be rephrased as follows, see \cite{LeMerdy98}: The Weiss conjecture is valid for $A$ and $A^{*}$ if and only if $A$ is similar to a contraction semigroup.

Versions of the Weiss conjecture for $Z=L^p$, $p\in[1,\infty)$ and more general spaces  have also been studied in the past, see e.g.\ \cite{BoDriEl10,haakthesis,HaLe05,HaKu07} (for the particular case of analytic semigroups). However, the somewhat `exotic' case $p=\infty$ has not gained a lot attention so far. To the best of the authors' knowledge, the only results in that direction are in \cite{BoDriEl10} and \cite{haakthesis} which imply that the Weiss conjecture for  $Z=L^{\infty}$ and any input space $U$ holds if and only if $A_{-1}$ itself is $L^{\infty}$-admissible, see Theorem \ref{thm:haakBounit}. However, we point out that the latter condition is very restrictive and still not fully understood,
see e.g.\ Proposition \ref{prop32}. 

\pagebreak[2]
Summarizing,  for the Weiss conjecture we  distinguish the following parameters:
\begin{itemize}
	\item the choice of $Z$: e.g., $Z=L^{p}$
	\item assumptions on the semigroup (bounded analytic, contraction,..)
	\item assumptions on the space $X$ (Hilbert space, reflexive,..)
	\item assumptions on the space $U$ ($\dim U<\infty$, $\dim U=\infty$)
\end{itemize}
Here, we will mainly consider bounded analytic semigroups. We will show that the assertion in Theorem  \ref{thm11} are equivalent to 
$$\text{{\it(iv)} The Weiss conjecture for }Z=L^{\infty}\text{ and any finite-dimensional }U\text{  holds true.}
 $$
 
The interest in $L^{\infty}$-admissibility {\color{black}and the existence of continuous solutions} comes from studying the notions of \textit{input-to-state stability}, well-known from finite-dimen\-sional system theory, that combine internal and external stability, in infinite-dimen\-sions. Recently, this subject has attained growing interest, see e.g.\ \cite{DaM13,JaLoRy08,KK2016IEEETac,Krstic16,MiI14b,MiWi17}. {\color{black} In particular, the existence of continuous solution is an axiom in the paper \cite{MiWi17}.} For linear systems \eqref{eq1}, input-to-state stability (ISS) is nothing else than exponential stability of the semigroup together with $L^{\infty}$-admissibility of $B$. The relation to so-called {\it integral input-to-state stability}, a variant of ISS, is more  involved. For systems of the form $\eqref{eq1}$, the following implication holds:
\begin{equation}\label{iss}
 \text{ integral input-to-state stability} \quad \implies \quad \text{input-to-state-stability}.
\end{equation}
In general, it is not known whether the converse holds in \eqref{iss}. Using the characterization of integral input-to-state stability in terms of admissibility derived in \cite{JNPS16}, 
in this paper, we show that it indeed holds in the situation of Theorem \ref{thm11}, which covers a broad class of applications. In particular, we prove (see Corollary \ref{cor22})
\begin{theorem}\label{thm:intr2}
The converse in \eqref{iss} holds for systems $\Sigma(A,B)$ provided that $B:U\to X_{-1}$ is bounded, $\dim U<\infty$ and $A$ generates an exponentially stable, analytic semigroup on a Hilbert space which is similar to a contraction semigroup.\end{theorem} 
Theorem \ref{thm:intr2}  generalizes results for parabolic diagonal systems derived in \cite{JNPS16}. 
We will further discuss how Theorems \ref{thm11} and \ref{thm:intr2} can be generalized to more general spaces $X$, and how the exponential stability can be weakened. 

In Section \ref{Sec2} we give sufficient conditions such that 
 continuity of mild solutions
 holds for the extremal set of  all input operator $B$ with finite-dimensional input space. Moreover, in this situation, we even obtain admissibility with respect to Orlicz spaces, which is a stronger property. This enables us to infer consequences for the converse of \eqref{iss}, Section \ref{Sec5}.
Section \ref{Sec3} deals with optimality of the conditions supposed in Secion \ref{Sec2} --- this is done by establishing the converse of the results  in Section \ref{Sec2} in terms of the $H^{\infty}$-calculus.
 
  In Section \ref{Sec4} we elaborate on the relation of the results to the Weiss conjecture. Finally, we conclude with an outlook, Section \ref{Sec6}, including a detailed  discussion of related (open) problems, 
 which may be of interest in their own right.

\subsection{Notions}
In the following we will always denote the generator of a $C_{0}$-semi\-group $T$ on a Banach space $X$ by $A$.  We will consider the spaces $X_{-1}^{A}$ and $X_{1}^{A}$ which are defined by the completion of $X$ using the norm $\|(\beta-A)^{-1}\cdot\|$ for some $\beta\in\rho(A)$ and by equipping $D(A)$ with the graph norm of $A$, respectively. If the operator $A$ is clear from the context, we will simply write $X_{-1}$ and $X_{1}$. By $R(\lambda,A)$ we denote the resolvent $(\lambda-A)^{-1}$.
If the dual operator $A^{*}:D(A^{*})\subset X'\to X'$ generates a strongly continuous semigroup --- that is, if $D(A^{*})$ is dense, e.g.\ when $X$ is reflexive ---
 then  
 $$(X_{-1}^{A})'\simeq X^{A^{*}}_{1}\quad\text{and}\quad(X_{1}^{A})'\simeq X_{-1}^{A^{*}},$$
   see e.g.\ \cite[p.~43-46]{vanNeerven} or \cite{Weiss89i}. If moreover $X$ is reflexive, we have that an element $x\in X_{-1}^{A}$ lies in $X$ if and only if the evaluation functional
\begin{equation}\label{eqfx}
f_{x}:X_{1}^{A^{*}}\rightarrow \C, \ y\mapsto \langle y, x\rangle_{X_{1}^{A^{*}}, X_{-1}^{A}},
\end{equation}
can be continuously extended to $X'$. We denote the  extension of $A$ to $X_{-1}^{A}$ by $A_{-1}$ and the $C_{0}$-semigroup generated by $A_{-1}$ by $T_{-1}$.

For Banach spaces $X,Y$, the bounded operators from $X$ to $Y$ will be denoted by $\mathcal{L}(X,Y)$. 
 A semigroup $T$ is called bounded analytic semigroup if it can be extended to a sector $\{0\}\cup\Sigma_{\alpha}$ where $\Sigma_{\alpha}=\{z\in\mathbb{C}\setminus\{0\}:|\arg(z)|<\alpha\}$, $\alpha\in(0,\pi]$, such that $T$ is bounded and analytic on $\Sigma_{\alpha}$. By an \textit{exponentially stable analytic semigroup} we refer to bounded analytic semigroup which in addition is exponentially stable on $[0,\infty)$.
 There is a natural correspondence between bounded analytic semigroups and sectorial operators. In fact, bounded analytic semigroups are characterized by the property that there exists a $\omega\in(0,\frac{\pi}{2})$ such that
 $$\sigma(-A)\subset \{0\}\cup\Sigma_{\omega}  \quad \text{and }\quad \sup_{z\in \C\setminus \overline{\Sigma_{\omega'}}}\|z (z+A)^ {-1}\|<\infty,\quad \forall\omega'\in(\omega,\pi).$$
 Operators $-A$ of the latter form are called \textit{sectorial} (of angle less than $\frac{\pi}{2}$).
For sectorial operators $-A$, the holomorphic functional calculus is a well-studied subject. Very roughly speaking, this calculus is a way to make sense of ``$f(-A)$'' for scalar-valued functions that are holomorphic on a domain that ``strictly contains'' the spectrum of $-A$. This is done by using an operator version of the Cauchy formula, the Riesz-Dunford integral,
$$f(-A)=\frac{1}{2\pi i}\int_{\partial \Sigma_{\omega'}}f(z)(\lambda+A)^{-1}\,dz,$$
where $\omega<\omega'<\theta$ and $f$ is a bounded analytic function on $\Sigma_{\theta}$ decaying suitably at $0$ and $\infty$. This construction can be extended to the whole of $H^{\infty}(\Sigma_{\theta})$, the space of bounded analytic functions on $\Sigma_{\theta}$, but will in general lead to unbounded operators $f(-A)$. If indeed $f(-A)\in \mathcal{L}(X)$ for all $f\in H^{\infty}(\Sigma_{\theta})$, then the $H^{\infty}(\Sigma_{\theta})$-calculus is called \textit{bounded}. For a detailed description of the construction we refer to the excellent monograph by Haase \cite{HaaseBook}.
 When dealing with analytic semigroups, it is controversial whether $``A$'' or $``-A$'' denotes the generator --- the latter being common in the study of maximal regularity. In this paper, we have decided to stick to the convention ``$A$'' as this is the usual choice in systems theory. 

{\color{black} Let $U$ be a Banach space and $I\subset[0,\infty)$ be an interval. In this paper the function space $Z(I;U)$ --- using the notation $Z(I)=Z(I;\C)$ for $U=\C$ --- will always refer to either of the following Banach spaces of $U$-valued functions: the continuous functions ${\rm C}(I;U)$, $L^{\infty}(I;U)$, $L^{1}(I;U)$ or an $U$-valued Orlicz space $E_{\Phi}(I;U)$ with respect to the Lebesgue measure on the interval $I$. This includes the $L^{p}$-spaces with $p\in(1,\infty)$. Let us briefly introduce the spaces $E_{\Phi}(I;U)$ --- a detailed exposition may e.g.\ be found in \cite{Kufner}. In what follows, let $\Phi:[0,\infty)\to[0,\infty)$ be a Young function, i.e.\ $\Phi$ is continuous, convex and increasing with $\lim_{x\to0}\frac{\Phi(x)}{x}=0$ and $\lim_{x\to\infty}\frac{\Phi(x)}{x}=\infty$. Then the set of Bochner measurable functions $f:I\to U$ such that there exists $k>0$ with $\Phi(\|kf(\cdot)\|)$ integrable becomes a Banach space, denoted by $L_{\Phi}(I;U)$, with the norm
\begin{equation}\label{orlicznorm}
\|u\|_{L_{\Phi}(I;U)}=\inf\left\{k>0\colon \int_{I}\Phi\left(\frac{\|u(s)\|}{k}\right)\,ds\leq 1\right\}.
\end{equation}
Although $L_{\Phi}$ is commonly referred to as {\it Orlicz space} in the literature, we rather would like to call the closure of boundedly   supported $L^{\infty}(I;U)$ functions,
$$E_{\Phi}(I;U)=\overline{\{f\in L^{\infty}(I;U)\colon {\rm ess~supp}(f) \text{ bounded}\}}^{\|\cdot\|_{L_{\Phi}(I;U)}},$$
the {\it Orlicz space with Young function $\Phi$} and let $\|\cdot\|_{E_{\Phi}(I;U)}=\|\cdot\|_{L_{\Phi}(I;U)}$. Note that the inclusion $E_{\Phi}(I;U)\subset L_{\Phi}(I;U)$ is strict in general, but, for example,  becomes an equality  in the case of  $L^{p}$ spaces, $p\in(1,\infty)$. Also note that for any $u\in E_{\Phi}(I;U)$, $\Phi\circ\|u(\cdot)\|$ is integrable.
}
Note that the following definition of {\it $Z$-admissible operators} particularly covers the common case of $L^{2}$-admissible operators.
\begin{definition}[Admissibility]\label{def:adm}
Let $U,Y$ be Banach spaces and let $Z$ be either $E_{\Phi}$, $L^{1}$, $L^\infty$ or ${\rm C}$  
(see the notation above). 
 \begin{enumerate}\item An operator $B\in \mathcal{L}(U,X_{-1})$ is called {\it (finite-time) $Z$-admissible (control) operator} for $A$ if, for all $t>0$, the operator
 \begin{align*}
 \Phi_{t}:={}&\Phi_{t}^{B}: Z(0,t;U)\rightarrow X_{-1}, u\mapsto \int_{0}^{t}T_{-1}(s)Bu(s)\,ds\\
 \intertext{has range in $X$, i.e. $\ran \Phi_{t}\subset X$ and is thus bounded from $Z(0,t;U)$ to $X$.
  \item An operator $C\in\mathcal{L}(X_{1},Y)$ is called (finite-time) $Z$-admissible observation operator for $A$ if, for all $t>0$,
 }
 \Psi_{t}:={}&\Psi_{t}^{C}:X_{1}\rightarrow Z(0,t;Y), x\mapsto CT(\cdot)x
  \end{align*}
  has a bounded extension to $X$, which we denote again by $\Psi_t$. 
  \end{enumerate}
  If  $\sup_{t>0}\|\Phi_{t}\|<\infty$ or $\sup_{t>0}\|\Psi_{t}\|<\infty$, then $B$  or $C$ is called infinite-time $Z$-admissible, respectively, in which case we write $\|B\|_{adm}:=\sup_{t>0}\|\Phi_{t}\|_{\mathcal{L}(Z,X)}$ and $\|C\|_{adm}:=\sup_{t>0}\|\Psi_{t}\|$.
    \end{definition}
    {\color{black}
Note that $B\in\mathcal{L}(U,X_{-1})$ is $Z$-admissible if and only if $\ran\Phi_{t}\subset X$ for some $t>0$. This can be seen using the semigroup property and that $u\in Z(0,t;U)$ implies $u(\tau+\cdot)|_{[0,s]}\in Z(0,s;U)$ for all $\tau\geq0, s>0$ and where $u$ is identified with its extension on $\R$ by zero. It is easily seen that the latter holds true for our choices of $Z$.
}
 For exponentially stable $C_{0}$-semigroups, finite-time $Z$-admissibility is equivalent to infinite-time $Z$-admissibility, see \cite[Lemma 8]{JNPS16}. 
The following proposition confirms the intuition that the results which hold for input spaces $U=\C$ generalize to finite-dimensional spaces $U$ --- for choices like $Z=L^2$ this is folklore.

\begin{proposition}\label{prop:finitedim}
Let $U$ be a finite dimensional Banach space and let $Z$ be either  $E_{\Phi}$, $L^{1}$, $L^\infty$ or ${\rm C}$. Then for any semigroup generator $A$ and $B\in\mathcal{L}(U,X_{-1})$ it holds that
$B$ is {\color{black}(infinite-time)} $Z$-admissible if and only if $Bf$ is  {\color{black}(infinite-time)} $Z$-admissible for every $f\in U$.
\end{proposition}
\begin{proof}
Let  $t>0$.  For $f\in U$
, $u\mapsto u\cdot f$ is  continuous from $Z(0,t;\C)$ to $Z(0,t;U)$.
Thus, if  $\Sigma(A,B)$ is  (infinite-time) $Z$-admissible  then $\Sigma(A,Bf)$ is  (infinite-time) $Z$-admissible. 
In order to prove the converse implications, we choose a basis $e_1,\dots, e_n$ of $U$.
Assume that $\Sigma(A,Bf)$ is $Z$-admissible for every $f\in U$. Let $t>0$ and $u\in Z(0,t;U)$.
Then  $u$ can be written as $u=\sum_{k=1}^n u_k e_k$ with $u_k\in Z(0,t)$ for $k=1,\dots,n$ and
\begin{equation}\label{eq:prop4} \int_0^t T_{-1}(s)B u(s)\, ds = \sum_{k=1}^n \int_0^t T_{-1}(s)Be_k u_k(s)\, ds \in X,\end{equation}
which implies that  $\Sigma(A,B)$ is $Z$-admissible. 
{\color{black} That the implication also holds true for infinite-time admissibility, now follows from the uniform boundedness principle. In fact, through \eqref{eq:prop4}, pointwise boundedness of the family  $(\Phi_{t}^{B})_{t>0}\subset\mathcal{L}(Z(0,t;U),X)$ is implied by pointwise boundedness of $({\Phi}_{t}^{Be_{k}})_{t>0,k\in\{1,..,n\}}$, where the notation from  Definition \ref{def:adm} is used.
}
\end{proof}
Note that Proposition \ref{prop:finitedim} does not generalize to the case where $U$ is infinite-dimensional, see e.g.\ Proposition \ref{prop32} or \cite{JaZwa04} (for the special cases of $Z=L^{2}$).

 \section{Orlicz space admissibility for finite-dimensional input spaces}\label{Sec2}
As mentioned in the introduction, if $B$ is $L^p$-admissible with $p\in[1,\infty)$, then the mild solutions of \eqref{eq1} are continuous. The analogous result holds for $Z$-admissibility if $Z$ is some Orlicz space, see Proposition \ref{prop11} below.
In this section we give sufficient conditions for such $Z$-admissibility. 
Note that there exist operators that are $L^{\infty}$-admissible but not $L^{p}$-admissible for any $p<\infty$, \cite[Ex.~5.2]{JNPS16}.
\begin{proposition}\label{prop11}
Let $X,U$ be Banach spaces and let $A$ generate a $C_{0}$-semigroup on $X$. If $B:U\to X_{-1}$ is $E_{\Phi}$-admissible for some Young function $\Phi$, then $B$ is $L^{\infty}$-admissible and the mild solution $x:[0,\infty)\to X$ of \eqref{eq1} is continuous for any $u\in L_{loc}^{\infty}(0,\infty;U)$ and any initial value $x_{0}\in X$.
\end{proposition}
\begin{proof}
This follows from \cite[Prop.~2.4]{JNPS16} and the fact that $E_{\Phi}$-admissibility of $B$ implies that $B$ is zero-class $L^{\infty}$-admissibility, that is, $B$ is $L^{\infty}$-admissible and $\lim_{t\to0^{+}}\|\Phi_{t}\|=0$ with $\Phi_t$ from Definition \ref{def:adm}. The latter can be argued by \cite[Thm.~3.1 and Prop.~2.12]{JNPS16}.
\end{proof}
The next technical result is at the core of what follows.  Recall that although $E_{\Phi}$-admissibility always implies $L^{\infty}$-admissibility, the corresponding implication is no longer true for {\it infinite-time} admissibility, see e.g.\ \cite{NaSchwe17}.

\begin{proposition}\label{theo:FCISS}
Let $A$ generate a bounded $C_{0}$-semigroup $T$ on a Banach space $X$  and let $x_{0}\in X$. Suppose that 
$(-A_{-1})^{\frac{1}{2}}$ is an infinite-time $L^{2}$-admissible control operator, and
let $f=(-A)^{\frac{1}{2}}T(\cdot)x_{0}$ either satisfy
$$\text{
\begin{enumerate*}[label=(\alph*)]
\item $f\in L^{2}(0,\infty;X)$, \hspace{0.1cm} or \hspace{0.1cm}
\item $\|f(\cdot)\|_{X}^{2}\in L_{\Psi}(0,\infty)$  for some Young function $\Psi$.
\end{enumerate*}}$$
If {\it (a)} holds, then $B=A_{-1}x_{0}$ is infinite-time $L^{\infty}$-admissible. \\If {\it (b)} holds, then there exists a Young function $\Phi$ and $C>0$ such that 
\begin{equation*}
\left\|\int_{0}^{t}T_{-1}(s)A_{-1}x_{0} u(s)\,ds \right\|_{X} \leq C\|u\|_{E_{\Phi}(0,t)},\quad \forall u\in E_{\Phi}(0,t),t>0.
\end{equation*}
Thus $B=A_{-1}x_{0}\in\mathcal{L}(\C,X_{-1})$ is infinite-time $E_{\Phi}$-admissible.

\end{proposition}

\begin{proof}
Recall that $A$ even generates a bounded analytic semigroup since $(-A_{-1})^{\frac{1}{2}}$ is $L^{2}$-admissible, \cite[Prop.~2.7]{BoDriEl10}.  We first show that $B$  is (finite-time) $L^{\infty}$-admissible. Therefore, we need to prove that $\Phi_{t}^{B}u\in X_{-1}$ lies in $X$ for all $u\in L^{\infty}(0,t)$ for $t>0$.
 From either of {\it (a)} or {\it (b)} it follows that (the restriction of) $f$ lies in {\color{black}$L^{2}(0,\frac{t}{2};X)$  and thus also $u(2\cdot)f\in L^{2}(0,\frac{t}{2};X)$.  By  $L^{2}$-admissibility of  $(-A_{-1})^{\frac{1}{2}}$ this implies that 
 \begin{equation*}\label{eq:factor}
\Phi_{\frac{t}{2}}^{(-A_{-1})^{\frac{1}{2}}}\left(u(2\cdot)f\right)\in X.
\end{equation*}
By definition of $\Phi_{\frac{t}{2}}^{(-A_{-1})^{\frac{1}{2}}}$ and $f$, also using that $A_{-1}T_{-1}(s)=T_{-1}(s)A_{-1}$ on $X$, 
\begin{align*}
\Phi_{t}^{B}{u}={}&\int_{0}^{t}T_{-1}(s)Bu(s)\, ds\\
={}&-\int_{0}^{t}T_{-1}(\tfrac{s}{2})(-A_{-1})^{\frac{1}{2}}u(
 s)(-A)^{\frac{1}{2}}T(\tfrac{s}{2})x_{0}\,ds\\
 ={}& -\frac{1}{2}\int_{0}^{\frac{t}{2}}T_{-1}(s)(-A_{-1})^{\frac{1}{2}}u(2s)f(s)\,ds\\
={}&-\frac{1}{2}\Phi_{\frac{t}{2}}^{(-A_{-1})^{\frac{1}{2}}}\left(u(2\cdot)f\right)\in X.\\
\end{align*}
We conclude  that $B=A_{-1}x_{0}$ is $L^{\infty}$-admissible and that
\begin{align*}\notag
	\|\Phi_{t}^{B}u\|_{X}\leq{}& \frac{1}{2}\|\Phi_{\frac{t}{2}}^{(-A_{-1})^{\frac{1}{2}}}\| \|u(2\cdot)f\|_{L^{2}(0,\frac{t}{2};X)}\\
	\leq{}&\frac{1}{2\sqrt{2}}\|(-A_{-1})^{\frac{1}{2}}\|_{adm} \|uf(\tfrac{\cdot}{2})\|_{L^{2}(0,t;X)}. \label{eq:ne2}
\end{align*}
In case of {\it (a)}, that is, $f\in L^{2}(0,\infty;X)$, we can estimate  the right-hand side by $c\|u\|_{L^{\infty}(0,t)} \|f\|_{L^{2}(0,\infty;X)}$ for some $c>0$ independent of $t$. This  shows that $B=A_{-1}x_{0}$ is infinite-time $L^{\infty}$-admissible.\\
Now assume that {\it (b)} holds, hence $g:s\mapsto \|f(\tfrac{s}{2})\|^{2}$  lies in $L_{\Psi}(0,\infty)$.  
Hence, by H\"older's inequality for Orlicz spaces see e.g.\ \cite{JNPS16,Kufner},
\begin{align*}
\|u(\cdot)f(\tfrac{\cdot}{2})\|_{L^{2}(0,t;X)} \leq{}&\|u^2\|_{L_{\tilde{\Phi}}(0,t)}^{\frac{1}{2}}\, \|g\|_{L_{\Psi}(0,t)}^{\frac{1}{2}}\\
\leq{}&\, \|u\|_{L_{\Phi}(0,t)}\, \; \|g\|_{L_{\Psi}(0,\infty)}^{\frac{1}{2}},
\end{align*}}
where $\tilde\Phi$  and $\Psi$ are complementary Young functions and ${\Phi}(x)=\tilde\Phi(x^2)$, which is a Young function since it is the composition of two Young functions. 
Therefore,
\begin{equation*}
\exists C>0\ \forall t>0,u\in L^{\infty}(0,t):\quad \left\|\Phi_{t}^{B}u\right\|_{X} \leq C\|u\|_{L_{{\Phi}}(0,t)}=C\|u\|_{E_{{\Phi}}(0,t)}.
\end{equation*}
Thus, $B$ is infinite-time $E_{\Phi}$-admissible since $L^{\infty}(0,t)$ is dense in $E_{\Phi}(0,t)$.
\end{proof}
By an important property of Orlicz spaces, condition {\it (a)} always implies {\it (b)} in Proposition \ref{theo:FCISS}. This enables us to prove the main result of this section.
\begin{theorem}\label{cor:sqfctest}
Let $A$ generate a bounded $C_{0}$-semigroup $T$  on a Banach space $X$. Suppose that $(-A)^{\frac{1}{2}}$ is an infinite-time $L^{2}$-admissible observation operator and  that $(-A_{-1})^{\frac{1}{2}}$ is an infinite-time $L^{2}$-admissible control operator. \\Then for any $B\in\mathcal{L}(U,X_{-1})$  with $\dim U<\infty$ and $\ran B\subset \ran A_{-1}$,
 it holds that
\begin{enumerate}[label=(\roman*),ref=(\roman*)]
\item $B$ is infinite-time $L^{\infty}$-admissible, and \label{cor:sqfctestit1}
\item there exists a Young function $\Phi$ such that $B$ is infinite-time $E_{\Phi}$-admissible.  \label{cor:sqfctestit2}
\end{enumerate}
\end{theorem}
\begin{proof}
By Proposition \ref{prop:finitedim}, it suffices to consider $B\in \ran A_{-1}$. We choose $x\in X$ such that $B=A_{-1}x$. Since $(-A)^{\frac{1}{2}}$ is an infinite-time $L^{2}$-admissible observation operator, $f(s)=(-A)^{\frac{1}{2}}T(s)x \in L^{2}(0,\infty;X)$ and since $\|(-A)^{\frac{1}{2}}T(\cdot)x\|$ is bounded on $(\tau, \infty)$, for  all $\tau>0$, Lemma \ref{lem:app1} in the Appendix implies that $\|f(\cdot)\|_{X}^{2}\in L_{\Psi}(0,\infty)$ for some Young function $\Psi$.  Therefore, the assumptions of Proposition \ref{theo:FCISS}, particularly both {\it (a)} and {\it (b)}, are satisfied. Hence the assertions \ref{cor:sqfctestit1} and \ref{cor:sqfctestit2} follow.
\end{proof}
{\color{black}
It is well-known that the conditions on $T$ in Theorem \ref{cor:sqfctest} are naturally linked with an equivalent norm of the space $X$, \cite{LeMerdy2003}.  In fact, for $x\in X$ and $u(s)=(-A)^{\frac{1}{2}}T(s)x$, the infinite-time $L^{2}$-admissibility gives, for $t>0$,
$${\color{black}\|T(t)x-x\|_{X}}=\|\Phi_{t}^{(-A_{-1})^{\frac{1}{2}}}u\|_{X}\leq C_{1}\|\Psi_{t}^{(-A)^{\frac{1}{2}}}x\|_{L^{2}(0,t;X)}\leq C_{2}\|x\|,$$
where $C_{1}$ and $C_{2}$ do not depend on $t$.
If we additionally assume that $T$ is strongly stable, i.e. $\lim_{t\to\infty}T(t)x=0$ for all $x\in X$, then  
$x\mapsto \|\Psi_{t}^{(-A)^{\frac{1}{2}}}x\|_{L^{2}(0,\infty;X)}$ becomes an equivalent norm --- which is Hilbertian if $X$ is a Hilbert space --- with respect to which $T$ is contractive.  From a system theoretic view-point this equivalence of norms means nothing else than the property that the system
\begin{align*}
\dot{x}={}&Ax, \qquad x(0)=x_{0},\\
y={}&(-A)^{\frac{1}{2}}x
\end{align*}
is {\it exactly observable}.
We point out that since $T$ is a bounded analytic semigroup in this situation, strong stability of $T$ is equivalent to $A$ having dense range, see e.g.\ \cite[Cor.~III.3.17]{EisnerBook2010}. 
Moreover, it is known that infinite-time $L^{2}$-admissibility of $(-A)^{\frac{1}{2}}$ and $(-A_{-1})^{\frac{1}{2}}$ together with $A$ having dense range is equivalent to $A$ satisfying  square function estimates of the form \footnote{In the literature, the notion `$A$ satisfies square function estimates' typically  refers to the version of \eqref{eq:sqfct} where only the second inequality is satisfied.}  
\begin{equation}\label{eq:sqfct} 
 \forall \phi_{0}\in H_{0}^{\infty}(\Cm)\exists  k,K>0\, \forall x\in X \quad k\|x\|^{2}\leq\int_{0}^{\infty}\|\phi_{0}(tA)x\|^{2}\frac{dt}{t}\leq K\|x\|^{2},
\end{equation}
where  $H_{0}^{\infty}(\Cm)=\{f\in H^{\infty}(\Cm)\colon \exists c,s>0: \|f(z)|\leq  \frac{c|z|^{s}}{(1+|z|)^{2s}}\}$, see e.g.\ \cite{haakthesis,LeMerdy2003,mcintoshHinf}. Note that for $\phi_{0}(z)=(-z)^{\frac{1}{2}}e^{-z}$, the inequality in \eqref{eq:sqfct} is nothing else than the previously mentioned equivalence of the norms $\|\cdot\|$ and $ \|\Psi_{t}^{(-A)^{\frac{1}{2}}}\cdot\|_{L^{2}(0,\infty;X)}$. It can also be shown that $T$ is contractive with respect to any of the equivalent norms induced by $\phi_{0}$ via \eqref{eq:sqfct}. For Hilbert spaces $X$ there is a converse: Any generator, with dense range, of a bounded analytic semigroup that is similar to a contractive semigroup satisfies \eqref{eq:sqfct}. This is a consequence of McIntosh's theorem stating that square function estimates of the form \eqref{eq:sqfct} characterize a bounded $H^{\infty}$-calculus when $X$ is a Hilbert space \cite[Theorem 7.3.1]{HaaseBook} and results 
due to Le Merdy \cite{LeMerdy98}, and independently derived by Franks and Grabowski-Callier (see \cite{HaaseBook} and the references therein).
With this  and Proposition \ref{prop11}, Theorem \ref{cor:sqfctest} leads to the following corollary.
}
\begin{corollary}\label{cor1}
Let $A$ generate a bounded analytic semigroup $T$ on a Hilbert space and suppose that $A$ has dense range. If $A$ has a bounded $H^{\infty}$-calculus, or equivalently if $T$ is similar to a contraction semigroup, then for every $B\in\mathcal{L}(U,X_{-1})$ with $\dim U<\infty$ and $\ran B \subset \ran A_{-1}$, 
\begin{enumerate}[label=(\roman*)]
\item there exists a Young function $\Phi$  such that $B$ is infinite-time $E_{\Phi}$-admissible. 
\item the mild solutions of \eqref{eq1} are continuous for any $u\in L_{loc}^{\infty}(0,\infty;U)$.
\end{enumerate}
\end{corollary}
By rescaling the results in this section can be adapted to analytic semigroups (instead of bounded analytic semigroups). Then, however, only finite-time admissibility is obtained in general.

\section{$L^{\infty}$-admissibility and bounded $H^\infty$-calculus}\label{Sec3}
The boundedness of the $H^{\infty}$-calculus, the crucial condition in Theorem  \ref{cor:sqfctest} and Corollary \ref{cor1}, may look artificially chosen in order to make the proofs, particularly of Proposition 6, work. The goal of this section is to demonstrate that this is not the case. In fact, we show that the converses of Theorem \ref{cor:sqfctest} and Corollary \ref{cor1} hold  for Hilbert spaces and explain what can be said in the case of more general Banach spaces. This reveals that the boundedness of the $H^{\infty}$-calculus appears naturally in this context. \\
It will be convenient to use the following notation for spaces of admissible operators.
\begin{align*}
\mathfrak{b}_{\infty}(A)&=\{B\in\mathcal{L}(\C,X_{-1})\colon B\text{ is infinite-time }L^{\infty}\text{-admissible for } A\}\\
\mathfrak{c}_{1}(A)&=\{C\in\mathcal{L}(X_{1},\C)\colon C\text{ is infinite-time }L^{1}\text{-admissible for } A\}.
\end{align*} 
Upon identification these spaces are contained in $X_{-1}$ and $(X_{1})'$, respectively. 
Note that under the conditions of Theorem \ref{cor:sqfctest}, it holds that $\mathfrak{b}_{\infty}(A)\supseteq\ran A_{-1}$. {\color{black}We will shortly see that this inclusion is in fact an equality.}

First we show that  analyticity of the semigroup is necessary under the condition that $\mathfrak{b}_{\infty}(A)=\ran A_{-1}$.

\begin{proposition}\label{prop:b(A)}
Let $A$ generate a semigroup $T$ on a  Banach space $X$. If $\frak{b}_{\infty}(A)=\ran A_{-1}$, then 
 $A$ generates a bounded analytic semigroup.
\end{proposition}
\begin{proof}
{\color{black}We first show that $T$ is bounded. For $x\in X$, $b=A_{-1}x$ is infinite-time $L^{\infty}$-admissible, hence with $u(s)=1$,
$$\sup_{t>0}\|T(t)x-x\|=\sup_{t>0}\left\|\int_{0}^{t}T_{-1}(s)A_{-1}xu(s)\,ds\right\|\leq \sup_{t>0} \|\Phi_{t}^{A_{-1}x}\|_{\mathcal{L}(L^{\infty}(0,t),X)}<\infty,$$
whence $T$ is bounded by the uniform boundedness principle.
}
With $u(s)=e^{-i\omega s-\varepsilon s}$, $\omega\in\mathbb{R}$, $\varepsilon>0$ we obtain for every $x\in X$
\begin{align*}
 \int_{0}^{\infty} T_{-1}(s)A_{-1}x u(s)\,ds &= \int_{0}^{\infty} e^{-(i\omega +\varepsilon) s}T_{-1}(s)A_{-1}x \,ds\\
 & = (i\omega+\varepsilon -A_{-1})^{-1}A_{-1}x \in X.
 \end{align*}
By the assumption that $\frak{b}_{\infty}(A)=\ran A_{-1}$, we get $\|(i\omega+\varepsilon-A_{-1})^{-1}A_{-1}x\| \leq C \|x\|_{X}$ for a constant $C$ independent of $x$, $\omega$ and $\varepsilon$. Thus,
\begin{align*}
\|i\omega(i\omega+\varepsilon-A)^{-1}x\|={}&\|(i\omega-(A_{-1}-\varepsilon))^{-1}(A_{-1}-\varepsilon)x+x\|\\
\leq{}& \|(i\omega-(A_{-1}-\varepsilon))^{-1}A_{-1}x\| + \varepsilon\|(i\omega-(A-\varepsilon))^{-1}x\|+\|x\|\\
\leq{}&(C+M+1)\|x\|,
\end{align*}
where $M\geq1$ is such that $\|(\lambda-A)^{-1}\|\leq \frac{M}{\Re \lambda}$ for all $\Re\lambda>0$.
Therefore, $A-\varepsilon$ generates a bounded analytic semigroup, see \cite{EngelNagel,Pazy83}. Moreover, since the sectorality constant $$\sup_{\Re z>0}\|z(z-(A-\varepsilon))^{-1}\|\leq C+M+1,$$
is bounded independently of $\varepsilon$, and thus
$$(1-A)^{-1}-(1-(A-\varepsilon))^{-1}=\varepsilon (1-(A-\varepsilon))^{-1}(1-A)^{-1}\to 0,$$
as $\varepsilon\to0$, the sequence of operators $(-A+\frac{1}{n})_{n\in\N}$ forms a sectorial approximation for $-A$  (on some sector $\Sigma_{\theta}$, $\theta\in(0,\frac{\pi}{2})$), see \cite[Sec.~2.1.2]{HaaseBook}. Thus $A$ also generates a bounded analytic semigroup.
\end{proof}
\begin{theorem}\label{thm2}
Let $A$ generate a semigroup on a Banach space $X$ and let $A$ have dense range.
Consider the following assertions. 
\begin{enumerate}[label=\textit{(\roman*)}, ref=\textit{(\roman*)}]
\item\label{thm2it1}  $\mathfrak{b}_{\infty}(A)=\ran A_{-1}$
\item\label{thm2it3} $A$ generates a bounded analytic semigroup $T$ and
 $-A$ has a bounded $H^{\infty}(\Sigma_{\theta})$-calculus for any $\theta\in(\frac{\pi}{2},\pi)$
\item\label{thm2it2}   $\mathfrak{c}_{1}(A)=\{ \langle y',A\cdot\rangle_{X',X}\colon y'\in X'\}$
\end{enumerate}
Then $\ref{thm2it1}\Rightarrow \ref{thm2it3}{ \color{black} \Leftrightarrow}\ref{thm2it2}$. 
If $X$ is reflexive, then also $\ref{thm2it3}\Rightarrow \ref{thm2it1}$.
\end{theorem}
Before we give a proof of Theorem \ref{thm2}, let us discuss the relation to the results in Section \ref{Sec2}.
{\color{black} Note that a bounded $H^{\infty}(\Sigma_{\theta})$-calculus for $\theta>\frac{\pi}{2}$ will in general not imply  a bounded $H^{\infty}(\Sigma_{\theta})$-calculus for any $\theta\leq\frac{\pi}{2}$ (as e.g.\ required in Corollary \ref{cor1}), see \cite{Kalton2003} for a counterexample on a reflexive space.}
However, for Hilbert spaces it is known that any generator of a bounded analytic semigroup has a bounded $H^{\infty}(\Sigma_{\theta})$-calculus for some $\theta<\frac{\pi}{2}$ if and only if the calculus is bounded on some sector, see \cite[Thm.~8, p.~225]{mcintoshHinf} for the original result by McIntosh or \cite[Thm.~7.3.1]{HaaseBook}. This gives the following characterization, Corollary \ref{cor3}. If $X$ is not a Hilbert space, we need an additional assumption:
From a fundamental result by Kalton and Weis \cite[Prop.~5.1]{KaltonWeisSums}, it follows that if $-A$ is $R$-sectorial of $R$-type $\omega_R<\frac{\pi}{2}$, i.e.\ $\sigma(-A)\subset {\overline\Sigma_{\omega_{R}}}$  and the set
$$\{\lambda R(\lambda,-A)\colon \lambda\in \overline{\Sigma_{\omega'}}^{c}\}\quad \text{is }R\text{-bounded}\quad \forall \omega'>\omega_{R},$$
then the boundedness of the $H^{\infty}$-calculus for some angle implies that the $H^{\infty}(\Sigma_{\theta})$-calculus is bounded for $\theta>\omega_{R}$.
Here, $R$-boundedness is a generalization of usual boundedness of sets of operators in $\mathcal{L}(X)$, see e.g.\ \cite{KaltonWeisSums} for definitions. Note that since $R$-boundedness coincides with boundedness in the operator-norm on Hilbert spaces, the notions of $R$-sectorial of $R$-type $\omega_{R}$ and sectorality of angle $\omega_{R}$ coincide then. Another example of an $R$-sectorial operator with $\omega_{R}<\frac{\pi}{2}$ is given by analytic contraction semigroups of positive operators on $L^p$ spaces with $p\in(1,\infty)$, see \cite[Cor.~5.2]{KaltonWeisSums}.
\begin{corollary}\label{cor3}
Let $A$ generate a bounded analytic semigroup on a Banach space $X$. 
Further assume that $-A$ is $R$-sectorial of angle $\omega_R\in(0,\frac{\pi}{2})$ --- which is particularly satisfied if $X$ is a Hilbert space --- and $A$ has dense range. Then the following assertions are equivalent.
\begin{enumerate}[label=(\roman*), ref=(\it\roman*)]
\item\label{cor21it1} $\mathfrak{b}_{\infty}(A)=\ran A_{-1}$.
\item\label{cor21it2} $-A$ has a bounded $H^{\infty}(\Sigma_{\theta})$-calculus for some $\theta<\frac{\pi}{2}$.
\item\label{cor21it3} $\mathfrak{c}_{1}(A)=\{ \langle y',A\cdot\rangle_{X',X}\colon y'\in X'\}$.
\end{enumerate}
\end{corollary}
\begin{proof}
{\color{black} It remains to show that \ref{cor21it2} implies \ref{cor21it1} (note that this automatically follows from Theorem \ref{thm2} if $X$ is reflexive) --- the rest follows from Theorem \ref{thm2} as argued above the corollary.  For that we use a standard argument for the $H^{\infty}$-calculus. For $u\in L^{\infty}(0,t)$ let $g\in H^{\infty}(\Sigma_{\theta})$, $\theta\in(0,\frac{\pi}{2})$, be defined by $g(z)=\int_{0}^{t}h(zs)u(s)\frac{ds}{s}$,  where $h(z)=ze^{-z}$. Since the $H^{\infty}(\Sigma_{\theta})$-calculus is bounded by assumption, we have that $g(-A)\in \mathcal{L}(X)$ and in particular 
$$g(-A)x=\int_{0}^{t}h(-sA)x\, u(s)\,\frac{ds}{s}=\int_{0}^{t}T(s)A_{-1}x\,u(s)\,ds\in X$$
for all $x\in X$. Since moreover $\|g(-A)\|\leq C\|g\|_{H^{\infty}({\Sigma_{\theta}})}\leq C_{\theta}\|u\|_{L^{\infty}(0,t)}$ where $C,C_{\theta}$ are independent of $t$, we conclude that $A_{-1}x\in \mathfrak{b}_{\infty}(A)$.}
\end{proof}
For the proof of Theorem \ref{thm2}, we need the following lemma.
\begin{lemma}\label{lem2}
Let $A$ generate a semigroup $T$ on a Banach space $X$.  If either $\mathfrak{b}_{\infty}(A)=\ran{A_{-1}}$ {\color{black} or $\mathfrak{c}_{1}(A)=\{ \langle y',A\cdot\rangle_{X',X}\colon y'\in X'\}$}, then $T$ is bounded analytic and there exists $C>0$ such that
\begin{equation}\label{eq4}
	\int_{0}^{\infty}|\langle y,AT(s)x\rangle|\, ds \leq C\,\|x\|\,\|y\|,\quad x\in X, y\in X'.
\end{equation}
\end{lemma}
\begin{proof}
{\color{black} Assume first that $\mathfrak{c}_{1}(A)=\{ \langle y',A\cdot\rangle_{X',X}\colon y'\in X'\}$. Then, by definition of $\mathfrak{c}_{1}(A)$, the mapping $y'\mapsto \Psi_{\infty}^{\langle y',A\cdot\rangle}x=\langle y',AT(\cdot)x\rangle_{X',X}$ is linear and well-defined from $X'$ to $L^{1}(0,\infty)$ for any $x\in X$.  An application of the closed graph theorem, yields that this is even a bounded operator  for every $x$ and thus, by the uniform boundedness principle, there exists $C>0$ such that  \eqref{eq4} holds. That $T$ is bounded analytic follows by a similar argument as in the proof of Proposition \ref{prop:b(A)}.
 }
Now assume that $\mathfrak{b}_{\infty}(A)=\ran{A_{-1}}$, which implies that $T$ is bounded analytic by Proposition \ref{prop:b(A)} and furthermore 
\begin{equation*}
\Psi_{u,t}:X\mapsto X, x\mapsto \int_{0}^{t}AT(s)x\, u(s)\,ds,
\end{equation*}
is well-defined for any $u\in L^{\infty}(0,t)$ and $t>0$. Let $x_{n}\to x$ and $\Psi_{u,t}x_{n}\to y$ in $X$. Since $\Psi_{u,t}x_{n}\to \Psi_{u,t}x$ in $X_{-1}$  by the boundedness of the integrand (bounded in $X_{-1}$), we conclude that $\Psi_{u,t}x=y$.  
Thus, $\Psi_{u,t}$ is bounded by the closed graph theorem. We further claim that $\Psi_{u,t}$ is  uniformly bounded for $t>0$ and $u\in \mathcal{U}_{t}=\{u\in L^{\infty}(0,t)\colon \|u\|_{\infty}=1\}$. In fact, for fixed $x\in X$ it follows again by $\mathfrak{b}_{\infty}(A)=\ran A_{-1}$ that $\sup_{u\in\mathcal{U}_{t}, t>0}\|\Psi_{u,t}x\| <\infty$. Hence, by the uniform boundedness principle, 
$C:=\sup_{u\in\mathcal{U}_{t},t>0}\|\Psi_{u,t}\| <\infty$.\newline
Let $x\in X$ and $y\in X'$. Define $u(s)={\mathrm{ exp}}(-i\:\mathrm{arg}\langle y,AT(s)x\rangle)$, $s\in(0,t)$ which lies in $\mathcal{U}_{t}$. 
With this, we obtain
\begin{align*}
	\int_{0}^{t}|\langle y,AT(s)x\rangle|\, ds = {}& \int_{0}^{t}\langle y, AT(s)x\rangle\, u(s)\ ds\\
									 = {}& \langle y, \int_{0}^{t} AT(s)x \, u(s) ds\rangle\\             
									 \leq {}& \|y\|\, \|\Psi_{u,t}x\|\\
									 \leq {}& C\, \|y\|\, \|x\|.\qedhere
\end{align*}
\end{proof}

\begin{myproof}{Theorem}{\ref{thm2}}  By Lemma \ref{lem2}, {\color{black} both \ref{thm2it1} or \ref{thm2it2}  respectively imply that $A$ generates a bounded analytic semigroup and that 
 \begin{equation}\label{eq9}
\exists C>0\quad\forall x\in X,y'\in X'\quad \|\langle y,AT(\cdot)x\rangle\|_{L^{1}(0,\infty)} \leq C\, \|x\|\,\|y\|.
 \end{equation}
 }
The latter condition is known as \textit{weak square function estimate} and first appeared in the seminal paper by Cowling, Doust, McIntosh and Yagi \cite{cowlingdoustmcintoshyagi}. By \cite[Cor.~4.5 and Ex.~4.8]{cowlingdoustmcintoshyagi}, it is equivalent to $-A$ having a bounded $H^{\infty}({\Sigma_{\theta}})$-calculus for $\theta>\frac{\pi}{2}$. Thus, \ref{thm2it1}$\Rightarrow$\ref{thm2it3}  and \ref{thm2it2}$\Rightarrow$\ref{thm2it3}. \newline
{\color{black}
By definition of $\mathfrak{c}_{1}(A)$, \ref{thm2it3}  implies via \eqref{eq9} that $\mathfrak{c}_{1}(A)\supseteq \{\langle y',A\cdot\rangle\colon y'\in X'\}$. To see the other inclusion, let $c\in\mathfrak{c}_{1}(A)$ and define $y'\in X'$ by $\langle y',x\rangle_{X',X}=-\int_{0}^{\infty}\langle c,T(s)x\rangle\,ds$ for $x\in X_{1}$. Using that $\lim_{t\to\infty}T(t)x=0$, which follows by analyticity and as $A$ has dense range, one easily deduces that $\langle c,x\rangle = \langle y',Ax\rangle$ for $x\in D(A^{2})$ and thus all $x\in X_{1}$. This shows that\ref{thm2it3}$\Rightarrow$\ref{thm2it2}. \\
As discussed above, \ref{thm2it3} is equivalent to \eqref{eq9} which also implies that $\mathfrak{c}_{1}(A^{*})\supseteq\{\langle x,A^{*}\cdot\rangle_{X,X'}\colon x\in X\}$.
 If $X$ is reflexive, then the latter inclusion of sets is even an equality by a similar argument as for the implication \ref{thm2it3}$\Rightarrow$\ref{thm2it2}. By \cite[Thm.~6.9(ii)]{Weiss89i}, using again reflexivity, we conclude that $\mathfrak{b}_{\infty}(A)=\ran A_{-1}$ --- in fact, for $b\in X_{-1}$, we have that $b\in \mathfrak{b}_{\infty}(A)$ if and only if $b^{*}\in\mathfrak{c}_{1}(A^{*})$ by \cite[Thm.~6.9(ii)]{Weiss89i}. Hence  \ref{thm2it3}$\Rightarrow$\ref{thm2it1}.
}
\end{myproof}

\medskip

In the proof of Theorem \ref{thm2} weak square function estimates seem to be the right choice to characterize bounded $H^{\infty}$-calculus. However, these are somehow `exotic' compared to the classic square functions (in the context of functional calculus). We refer to \cite{hahaSQFCT} for a detailed discussion of their relations.

\begin{rem}
As we have seen, the condition $\mathfrak{b}_{\infty}(A)=X_{-1}$, i.e. $L^{\infty}$-admissibility for any possible scalar input, can be rewritten as
\begin{equation} \label{eq5}
\int_{0}^{t}AT(s)x u(s) \,ds\in X, 
\end{equation}
for all  $x\in X$, and $u\in L^{\infty}(0,t)$. As $\|AT(s)\|$ behaves like $s^{-1}$ for analytic semigroups,  \eqref{eq5} can be seen as a condition on the convergence of a singular integral. In other words, it is a `unconditionality' condition, which is a natural phenomenon for a bounded $H^{\infty}$-calculus,  see \cite[Theorem 5.6.2]{HaaseBook}.
\end{rem}

\section{The Weiss conjecture and some counterexamples}\label{Sec4}

 Corollary \ref{cor1} shows that for bounded analytic semigroups on Hilbert spaces with bounded $H^{\infty}$-calculus the set of $L^{\infty}$-admissible operators with finite-dimensional input space $U$ is as large as possible --- it equals $\mathcal{L}(U,X_{-1})$ if additionally $0\in\rho(A)$. 
 One may ask what does happen for infinite-dimensional spaces $U$. For that, let us draw the connection to the \textit{Weiss-conjecture} for control operators:
{\color{black} 
 As indicated in the introduction, G.~Weiss 
 formulated the problem whether infinite-time $L^{2}$-admissibility follows already from a necessary condition on the resolvent, which is derived considering $\Phi_{\infty}u$ with $u(s)=e^{-\lambda s}u$, $u\in U$, {$\Re\lambda>0$}, see Definition \ref{def:adm}. By using the Laplace transform and H\"older's inequality, this yields that infinite-time $L^{p}$-admissibility, $p\in[1,\infty)$, implies that 
 $$\sup_{\Re\lambda>0}\|(p\Re \lambda)^{\frac{1}{p}}R(\lambda,A_{-1})B\|_{\mathcal{L}(U,X)}<\infty.$$
 For $p=\infty$, this gives
 \begin{equation}\label{eq:rescond}
 \sup_{\Re\lambda>{0}}\|R(\lambda,A_{-1})B\|_{\mathcal{L}(U,X)}<\infty.
 \end{equation}
 }
In analogy to $\mathfrak{b}_{\infty}(A)$, let us introduce the sets
 $$\mathfrak{B}_{\infty,U}(A):=\{B\in\mathcal{L}(U,X_{-1})\colon B\text{ is infinite-time }L^{\infty}\text{-admissible}\}.$$
This means that the Weiss conjecture for $p=\infty$ asks whether 
 $\mathfrak{B}_{\infty,U}(A)$  equals $\{B\in\mathcal{L}(U,X_{-1})\colon \sup_{\Re\lambda>{0}}\|R(\lambda,A_{-1})B\|_{\mathcal{L}(U,X)}<\infty\}$. 
 {\color{black} Note that if the semigroup is bounded analytic, then \eqref{eq:rescond} is satisfied for any $B=b\in \ran A_{-1}$. Thus, by Corollary \ref{cor3} and Proposition \ref{prop:finitedim}, the Weiss conjecture for $p={\infty}$ and finite-dimensional $U$ has an affirmative answer for generators $A$ on Hilbert spaces if the following conditions are met
 \begin{itemize}
 \item  $A$ generates a bounded analytic semigroup and has dense range,
 \item  $A$ has a bounded $H^{\infty}$-calculus.
 \end{itemize}
 For a converse implication see Remark \ref{rem17} below.}
  However, for infinite-dimensional $U$ this cannot be expected in general as the following result shows.

\begin{proposition}\label{prop32}
Let $X$ be a Hilbert space with orthonormal basis $\{e_n\}$ and let $(\lambda_n)_{n\in \mathbb N}\subset \mathbb C_-$ such that  $(\Re \lambda_n)_{n\in \mathbb N}$ is a  monotonically decreasing sequence with $\lim_{n\rightarrow \infty}\Re \lambda_n=-\infty$ and $|\mbox{\rm Im} \,\lambda _n|\le k |\Re \lambda_n|$ for some $k>0$ and all $n\in \mathbb N$. Let $A:D(A)\subset X\rightarrow X$ be given by
\[ Ae_n:= \lambda_n e_n, \qquad D(A):=\{ x=\sum_{n} x_n e_n\mid \sum_n |\lambda_n x_n|^2 <\infty\}.\]
Further, we define $U:=X$ and $B\in L(U,X_{-1})$ by $B:=A_{-1}$. Then we have:
\begin{enumerate}[label=(\roman*)]
\item\label{prop3.1:it1} For every $u\in U$, the system $\Sigma(A,Bu)$ is $L^{\infty}$-admissible. 
\item\label{prop3.1:it2} $\Sigma(A,B)$ is not $L^{\infty}$-admissible.
\item\label{prop3.1:it3} $\Sigma(A,B)$ satisfies the resolvent condition \eqref{eq:rescond}. 
\end{enumerate}
\end{proposition}
\begin{proof}
The operator $A$ generates an exponentially stable analytic diagonal semigroup on $X$. Thus \ref{prop3.1:it1} follows from Corollary \ref{cor1} and \ref{prop3.1:it3} is clear by the resolvent identity and since $-A$ is sectorial. Thus it remains to prove \ref{prop3.1:it2}. We choose a subsequence $(\gamma_n)_{n\in\mathbb N}$ of $(\lambda_n)_{n\in\mathbb N}$ such that
\[ \Re  \gamma_1\le -1 \quad \mbox{ and } \quad \Re \gamma_{n+1}< 2 \Re \gamma_n, \qquad n\in \mathbb N.\]
Let $u\in L^\infty(0,1;X)$ be given by $u(s)=\sum (u(s))_n e_n$ with
\[ (u(s))_n :=\begin{cases} 0, & \lambda_n\not=\gamma_m \mbox{ for all } m\in \mathbb N,\\
1, &  \lambda_n=\gamma_m  \mbox{ for some } m\in \mathbb N \mbox{ and } s\in \left[-\frac{1}{2\Re \gamma_m},-\frac{1}{\Re \gamma_m}\right].
\end{cases}\]
It is easy to see that $u\in L^\infty(0,1;X)$ with $\|u\|=1$. However,
\begin{align*}
\left\|\int_0^1 T_{-1}(s)Bu(s)\,ds\right\|^{2} & = \sum_{n=1}^\infty \left|\int_0^1 e^{\lambda_n s} \lambda_n (u(s))_n\, ds \right|^2\\
&= \sum_{m=1}^\infty \left( e^{-1/2}-e^{-1}\right)^2 \\
&=\infty,
\end{align*}
which shows that  $\Sigma(A,B)$ is not $L^\infty$-admissible. Note, however, that $A$ has a bounded $H^{\infty}$-calculus as the generated semigroup is a contraction semigroup.
\end{proof}
Proposition \ref{prop32} also shows that Proposition \ref{prop:finitedim} fails for infinite-dimensional spaces $U$ and $Z=L^{\infty}$, and hence shows that the so-called \textit{weak Weiss conjecture} is not true in this situation. Moreover, the example is in line with the following known characterization of when the $L^{\infty}$-Weiss conjecture actually holds. For reflexive spaces $X$ this follows already from the dual situation of $L^{1}$-admissible operators in \cite[Sektion 2.3]{haakthesis} and  Bounit--Driouich--El-Mennaoui {\cite{BoDriEl10} by using \cite[Thm.~6.9]{Weiss89i}. For completeness we provide a proof of the general case.
{\color{black}
\begin{theorem}\label{thm:haakBounit}
Let $A$ generate a  semigroup $T$ on a Banach space $X$. Then the following assertions are equivalent.
\begin{enumerate}[label=(\it\roman*),ref=(\it\roman*)]
\item\label{it:WC2} $A_{-1}$ is infinite-time $L^{\infty}$-admissible,
\item\label{it:WC3} $T$ is bounded analytic and the $L^{\infty}$-Weiss conjecture for $A$ holds, i.e.
$$\mathfrak{B}_{\infty,U}(A)=\{B\in\mathcal{L}(U,X_{-1})\colon\sup_{\lambda>{0}}\|R(\lambda,A_{-1})B\|_{\mathcal{L}(U,X)}<\infty\}$$
 for all Banach spaces $U$.
\end{enumerate}
If additionally $0\in\rho(A)$, then \ref{it:WC2} and \ref{it:WC3} are further equivalent to 
\begin{enumerate}[resume,label=(\it\roman*),ref=(\it\roman*)]
\item\label{it:WC1} $T$ is bounded analytic and
$\mathfrak{B}_{\infty,U}(A)=\mathcal{L}(U,X_{-1})$ for all Banach spaces $U$.
\end{enumerate}
\end{theorem}
\begin{proof}
By the resolvent identity, it follows easily that $\|R(\lambda,A_{-1})A_{-1}\|$ is uniformly bounded for $\Re\lambda>0$ if $T$ is bounded analytic. Hence, \ref{it:WC3} implies \ref{it:WC2} and if $0\in\rho(A)$, then \ref{it:WC3} is equivalent to \ref{it:WC1}. Hence it remains to show \ref{it:WC2}$\implies$\ref{it:WC3}. 
We assume that $A_{-1}$ is infinite-time $L^{\infty}$-admissible.  Then $\mathfrak{b}_{\infty}(A)=\ran A_{-1}$ and thus $T$ is bounded analytic by Proposition \ref{prop:b(A)}.  Let $B\in\mathcal{L}(U,X_{-1})$ for some Banach space $U$ with $C:=\sup_{\lambda>0}\|R(\lambda,A_{-1})B\|_{\mathcal{L}(U,X)}<\infty$. For $\lambda>0$ we can write $B=(\lambda-A_{-1})B_{\lambda}$ for $B_{\lambda}=(\lambda-A_{-1})^{-1}B\in\mathcal{L}(U,X)$. Therefore, for $u\in L^{\infty}(0,t;U)$,
$$\int_{0}^{t}T_{-1}(s)Bu(s)\,ds=\lambda\int_{0}^{t}T(s)B_{\lambda}u(s)\,ds-\int_{0}^{t}T_{-1}(s)A_{-1}B_{\lambda}u(s)\,ds\in X$$
by the assumption that $A_{-1}$ is $L^{\infty}$-admissible and since $B_{\lambda}u(\cdot)\in L^{\infty}(0,t;X)$. Choosing $\lambda=\frac{1}{t}$, 
\begin{align*}\left\|\int_{0}^{t}T_{-1}(s)Bu(s)\,ds\right\|\leq{}& \left(\sup_{s>0}\|T(s)\|+\|\Phi_{t}^{A_{-1}}\|\right)\|B_{t^{-1}}\| \,\|u\|_{L^{\infty}(0,t;U)}\\
\leq{}& \left(\sup_{s>0}\|T(s)\|+\|\Phi_{t}^{A^{-1}}\|\right)C\,\|u\|_{L^{\infty}(0,t;U)}
\end{align*}
where we used that $\|B_{t^{-1}}\|=\|R(t^{-1},A_{-1})B\|\leq C$. Since $A_{-1}$ is infinite-time admissible, $\sup_{t>0}\|\Phi_{t}^{A_{-1}}\|_{\mathcal{L}(L^{\infty}(0,t;U),X)}<\infty$ and we conclude that $B$ is infinite-time $L^{\infty}$-admissible. 
\end{proof}
}

We emphasize that $A_{-1}$ being infinite-time $L^{\infty}$-admissible is a very strong and restrictive condition; much stronger than $A$ having a bounded $H^{\infty}$-calculus as Proposition \ref{prop32} shows. Moreover, it is an open question if the condition already implies that $A$ is a bounded operator. 
{\color{black}
Under an even (slightly) stronger condition, this can indeed be proved, as the following result shows.
} For that recall the following refinement of admissibility:
a $Z$-admissible operator $B$ is called zero-class $Z$-admissible if
\[\sup_{\|u\|_{Z(0,t;U)}\leq 1}\left\|\int_{0}^{t} T_{-1}(s)Bu(s)\,ds\right\| \rightarrow 0,\quad \text{as }t\searrow0.
\]

{\color{black}
\begin{proposition}\label{prop33}
Let $A$ generate a $C_{0}$-semigroup $T$ and let $Z$ be ${\rm C}$, $L^{\infty}$, $L_{1}$ or $E_{\Phi}$. Then $A_{-1}$ is a zero-class $Z$-admissible operator if and only if $A$ is bounded.
\end{proposition}
\begin{proof}
If $A$ is bounded, then $A_{-1}$ is easily seen to be zero-class $Z$-admissible for any of the considered choices of $Z$. To show the converse, note that it suffices to consider the case where $Z$ refers to the continuous functions ${\rm C}$ as ${\rm C}(I)$ is embedded in $L^{1}(I)$ and $E_{\Phi}(I)$ for bounded intervals $I$. For $x\in X$, $\|x\|=1$ choose the constant function $u(\cdot)=x$. Then 
$$\|T(t)x-x\|_{X} = \left\| \int_{0}^{t}AT(s)u(s)\,ds\right\|_{X} \leq \|\Psi_{t}\|\|u\|_{L^{\infty}(0,t;U)}=  \|\Psi_{t}\|,$$
where $\Psi_{t}$ is defined as in Definition \ref{def:adm}. Since $\|\Psi_{t}\|\to0^{+}$ by zero-class admissibility, $T$ is uniformly continuous and hence $A$ is bounded (see e.g.\ \cite[Cor.II.1.5]{EngelNagel}).
\end{proof}
}
Proposition \ref{prop33} shows that if we could find an $L^{\infty}$-admissible $A_{-1}$ such that $A$ is unbounded, 
then $A_{-1}$ is not zero-class $L^{\infty}$-admissible and thus not $E_{\Phi}$-admissible for any Young function $\Phi$. 
We remark that in order to show the existence of such an $A_{-1}$, it suffices to find an unbounded 
generator $A$ on a reflexive space such that $A$ is an $L^{1}$-admissible observation operator. This follows from the duality results in \cite{Weiss89i} mentioned above. The difficulty is the reflexivity of the space --- examples of unbounded $L^{1}$-admissible observation operators $A$ for $X=\ell^{1}(\mathbb{N})$ can be easily constructed by diagonal operators. 
\begin{rem}\label{rem17}
Let $A$ generate a bounded analytic semigroup $T$ on a a Hilbert space $X$ and assume that $A$ has dense range. Setting our results, in particular Theorem \ref{thm2}, in context with Le Merdy's characterization of when the $L^2$-Weiss conjecture holds true, \cite{LeMerdy2003}, we arrive at the following list of equivalent conditions:
\begin{enumerate}[label=(\roman*)]
\item The $L^{2}$-Weiss conjecture holds true for $A$ and $A^{*}$ and any space $U$.
\item $T$ is similar to a contraction semigroup.
\item The $L^{\infty}$-Weiss conjecture holds true for $A$ and any finite-dimensional $U$.
\end{enumerate}
In particular this shows that the Weiss conjecture ``does not extrapolate'' in the sense that its validity for $L^2$ does not imply the validity for $L^{\infty}$.
\end{rem}

\section{Applications to input-to-state stability}\label{Sec5}
The results of Sections \ref{Sec2} and \ref{Sec3} have direct consequences for notions of input-to-state stability, because of their characterization via admissibility, see Theorem \ref{thmjnps} below. For the sake of completeness, we include the definitions for which we need the following function classes commonly used in Lyapunov theory.
\begin{align*} 
\calK ={}& \{\mu \colon [0,\infty)\rightarrow [0,\infty) \:|\: \mu(0)=0,\, \mu\text{ continuous, }  \text{strictly increasing}\},\\
 \calK_\infty ={}& \{\theta\in\calK \:|\:  \lim_{x\to\infty} \theta(x)=\infty\},\\
\mathcal{L}={}&\{\gamma \colon [0,\infty)\rightarrow[0,\infty)\:|\:\gamma \text{ continuous, } \text{strictly decreasing,} \lim_{t\to\infty}\gamma(t)=0 \},\\
\mathcal{KL} = {}&\{ \beta \colon [0,\infty)^{2}\rightarrow[0,\infty)\: | \: \beta(\cdot,t)\in\calK\ \forall t \geq 0 \text{ and }\beta(s,\cdot)\in\mathcal{L}\ \forall s> 0\}.
\end{align*}

 \begin{definition}[Input-to-state stability]
 A system $\Sigma(A,B)$ of form \eqref{eq1} is called 
 \begin{itemize}
 \item {\em input-to-state stable with respect to  $Z$} (or {\em  $Z$-ISS}), if there exist functions $\beta\in \mathcal{KL}$ and $ \mu\in {\mathcal K}_\infty$ such that for every $t\ge 0$, $x_0\in X$ and  $u\in Z(0,t;U)$,
  $$x(t) \text{ lies in }X\text{ and }
\left\| x(t)\right\| \le \beta(\|x_0\|,t)+  \mu (\|u\|_{Z(0,t;U)}).$$

\item {\em integral input-to-state stable (integral ISS)} if there exist functions  $\beta\in \mathcal{KL}$, $\theta\in {\mathcal K}_\infty$ and $\mu \in {\mathcal K}$ such that for every $t\ge 0$, $x_0\in X$ and  $u\in L^{\infty}(0,t;U)$,
  $$x(t) \text{ lies in }X\text{ and }
\left\| x(t)\right\| \le \beta(\|x_0\|,t)+ \theta \left(\int_0^t \mu (\|u(s)\|_{U})ds\right).$$
\end{itemize}
 \end{definition}

 The following result shows that in the case of linear systems (integral) ISS can be characterized by exponential stability and admissibility in certain norms.
 \begin{theorem}[\protect{\cite{JNPS16}}]\label{thmjnps}
A system $\Sigma(A,B)$ is 
\begin{itemize}
\item $Z$-ISS  if and only if $A$ generates an exponentially stable semigroup and $B$ is $Z$-admissible.
\item integral ISS if and only if $A$ generates an exponentially stable semigroup and $B$ is $E_{\Phi}$-admissible for some Young function $\Phi$.
\end{itemize}
 \end{theorem}
 For a generalization of this theorem for non-exponentially stable semigroups, see \cite{NaSchwe17}. Using Theorem \ref{thmjnps}, Proposition \ref{prop:finitedim} has the following version.
\begin{proposition}\label{prop2:finitedim}
Let $U$ be a finite dimensional Banach space and let $Z$ be either $E_{\Phi}$, $L^1$  or $L^\infty$. Then we have for any  generator $A$ and $B\in\mathcal{L}(U,X_{-1})$ that
\begin{enumerate}[label=(\roman*)]
\item \label{prop:finitedim2}$\Sigma(A,B)$ is $Z$-ISS if and only if $\Sigma(A,Bf)$ is $Z$-ISS for any $f\in U$,
\item $\Sigma(A,B)$ is integral ISS if and only if $\Sigma(A,Bf)$ is integral ISS for any $f\in U$.
\end{enumerate}
\end{proposition}
 With this we can formulate direct consequences of Theorem \ref{cor:sqfctest} and Corollary \ref{cor1}.
\begin{corollary}\label{cor22} Let $A$ generate an exponentially stable analytic semigroup $T$  on a Banach space $X$. If additionally either
\begin{enumerate}[label=(\roman*)]
\item $X$ is a Hilbert space and $T$ is similar to a contraction semigroup, or
\item  $(-A)^{\frac{1}{2}}$ and $(-A_{-1})^{\frac{1}{2}}$ are $L^{2}$-admissible (observation/control) operators,
\end{enumerate}
then $\Sigma(A,B)$ is integral ISS for any $B\in\mathcal{L}(U, X_{-1})$ with finite-dimensional $U$. In particular, the notions of integral ISS and ISS coincide for $\dim U<\infty$.
\end{corollary}

\begin{rem}
 Corollary \ref{cor22} generalizes Theorem 4.1 in \cite{JNPS16} in the case of exponentially stable parabolic diagonal systems on Hilbert spaces, that is, operators $A$ of the form $Ae_{n}=\lambda_{n}e_{n}$, $n\in\mathbb{N}$, where $(e_{n})_{n\in\mathbb{N}}$ is a Riesz-basis of $X$ and $(\lambda_{n})_{n\in\mathbb{N}}$ lie in a suitable sector.  In fact, the generated semigroup is similar to a contraction semigroup. Conversely, if $(\lambda_{n})_{n\in\mathbb{N}}$ is an \textit{interpolating sequence}, then the Riesz-property of the basis is implied by similarity to a contraction semigroup, see  \cite{LeMerdy2003}. This shows that the assumption of the Riesz-property in \cite{JNPS16} is not necessary in general.
 \end{rem}
\begin{rem}
Corollary \ref{cor22} can be generalized to the weaker notions of `strong input-to-state stability' and `strong integral ISS' which are discussed in \cite{NaSchwe17}. In these notions, exponential stability of the semigroup is replaced by strong stability and the result follows as above by using a  generalization of Theorem \ref{thmjnps} from \cite{NaSchwe17}.
\end{rem}

\section{Discussion and Outlook}\label{Sec6}
As described in the introduction, an open question is if integral ISS is always implied by ISS for linear systems \eqref{eq1}. As mentioned before, this can be rephrased as an operator-theoretic problem.
\begin{question}\label{Q0}
\textit{Is $B\in\mathcal{B}(U,X_{-1})$ an $E_{\Phi}$-admissible control operator for some  ${\Phi}$ provided that $B$ is an $L^{\infty}$-admissible control operator?}
\end{question} This in turn can be seen as the question whether for $B\in\mathcal{L}(U,X_{-1})$, the mapping $\Phi_{t}:L^{\infty}(0,t;U)\to X$, see Definition \ref{def:adm}, can always be extended to some Orlicz space $E_{\Phi}(0,t;U)$. 
We can also formulate the dual problem.
\begin{question}\label{Q11}
\textit{Is $C\in\mathcal{B}(X_{1},Y)$  an $E_{\Phi}$-admissible observation operator for some  ${\Phi}$ provided that $C$ is an $L^{1}$-admissible observation operator?}
\end{question}

In contrast to Question \ref{Q0}, we can provide a negative answer to Question \ref{Q11}.
\begin{example}\label{prop5}
Let $X=L^{1}(0,1)$ and $T$ be the left-shift semigroup on $X$. Then $C=\delta_{0}$ is  an $L^{1}$-admissible observation operator, but $C$ is not $E_{\Phi}$-admissible for any Young function $\Phi$.
\end{example}
\begin{proof}
Since for any $f\in L^{1}(0,1)$ and $s\leq1$, we have that
$CT(s)f=f(s)$,
it follows that $C$ is $L^{1}$-admissible and not $E_{\Phi}$-admissible as $L^{1}(0,t)\supsetneq E_{\Phi}(0,t)$ for any $\Phi$.
\end{proof}
Note that Example \ref{prop5} does not yield an answer for Question \ref{Q0}: in fact, if we had an example of a generator on a reflexive space together with an $L^1$-admissible $C$ which is not $E_{\Phi}$-admissible for any $\Phi$, then the dual semigroup and $B=C^{*}$ would provide a (negative) answer for Question \ref{Q0}.

In this article we have shown that  the answer to Question \ref{Q0} is `yes' if $U$ is finite-dimensional, the semigroup is bounded analytic and strongly stable (which is equivalent to $A$ having dense range here)  on a Banach space and $A$ satisfies (``two-sided'') square function estimates of the form \eqref{eq:sqfct}. The next step is of course to ask what happens without the latter assumption. In particular, we ask if the implication still holds true for any analytic semigroup on a Hilbert space.  It is well-known that on every (separable) Hilbert space, a bounded analytic semigroup with generator $A$ can be constructed such that $A$ does not have a bounded $H^{\infty}$-calculus, that is $A$  does not satisfy square function estimates of the above type, \cite{BaillonClement91,HaaseBook,mcintoshHinfNo}. The question remains whether for such operators the sets $\mathfrak{B}_{\infty,U}(A)$ and $$\mathfrak{B}_{\text{Orlicz},U}(A)=\{B\in\mathcal{L}(U, X_{-1})\colon B \text{ is }E_{\Phi}\text{-admissible for some }\Phi\}$$ still coincide for $U$ is finite-dimensional. By Proposition \ref{prop:finitedim}, it again suffices to consider $U=\C$.
\begin{question}\label{Q1}
Let $A$ generate an exponentially stable analytic semigroup on a Hilbert space. Suppose that the $H^{\infty}$-calculus for $A$ is not bounded.  Does Question \ref{Q0} have a positive answer, i.e.\ does the following implication hold
$$b\in X_{-1} \text{ is }L^{\infty}\text{-admissible} \quad \implies\quad \exists \Phi: b \text{ is }E_{\Phi}\text{-admissible} \quad?$$
\end{question}
By Theorem \ref{thm2}, it is clear that in the situation of Question \ref{Q1}, $\mathfrak{b}_{\infty}(A)\subsetneq X_{-1}$. Because of this strict inclusion, the space $\mathfrak{b}_{\infty}(A)$ is hard to characterize which makes it difficult to investigate the question. However, we have to emphasize that the condition of a bounded $H^{\infty}$-calculus is not very restricting in  practice when working with specific pde's as it is known to hold true for a large class of operators, including most differential operators, see e.g.\ \cite[Sec.~3]{WeisSurvey} for a survey. Therefore, Question \ref{Q1} is rather of theoretic interest.

On the other hand, when leaving the Hilbert space setting, there is a well-known subtlety concerning the relation of a bounded $H^{\infty}$-calculus and square function estimates. In general, classical square functions of the form \eqref{eq:sqfct}
are only sufficient but not necessary for a bounded calculus, see e.g.\ \cite{cowlingdoustmcintoshyagi}. This issue, however, can be overcome by using generalized square function estimates that were first introduced only for $L^{p}$ spaces and later, in a seminal paper by Kalton and Weis \cite{KaltonWeisUnpublished}, generalized to general Banach spaces. For $X=L^{p}(\Omega,\mu)$, $1\leq p<\infty$, the corresponding condition for $(-A)^{\frac{1}{2}}$ reads 
\begin{equation*}
\exists K>0,\forall f\in L^{p}(\Omega,\mu):\ \left\|\left(\int_{0}^{\infty}\left|((-A)^{\frac{1}{2}}T(t)f)(\cdot)\right|^{2}\,dt\right)^{\frac{1}{2}}\right\|_{L^{p}(\Omega,\mu)} \leq K\|f\|_{L^{p}(\Omega,\mu)},
\end{equation*}
which is consistent with  $L^2$-admissibility in the case $p=2$. It is not hard to see that this condition together with its dual version can be used  to derive that any $b\in X_{-1}$ is $L^{\infty}$-admissible along the same lines as in the proof of Theorem \ref{cor:sqfctest}. Moreover, one even gets zero-class $L^{\infty}$-admissibility. However, the difference is that it is not clear if an Orlicz-space norm can be recovered as in the case for classical $L^{2}$-admissibility. 
\begin{question}
Let $X=L^{p}(\Omega,\mu)$ and let $A$ generate an exponentially stable analytic semigroup such that $A$ has a bounded $H^{\infty}$-calculus.
Does the implication
$$b\in X_{-1} \quad \implies\quad \exists \Phi: b \text{ is }E_{\Phi}\text{-admissible} \quad$$
hold true?
\end{question}
In the case that $L^{p}(\Omega,\mu)=\ell^{p}(\N)$ this was proved in \cite[Thm.~4.1]{JNPS16} using a direct method without (explicitly) using the boundedness of the calculus. Also note that Fackler \cite{FacklerLp} recently provided an $L^p$-version of the Le Merdy--Grabowski--Callier result linking the boundedness of the $H^{\infty}$-calculus to a corresponding positive, contractive semigroup, see also \cite{Weis01}.

This article was mainly concerned with parabolic pde's, but it is an interesting question what happens for more general semigroups.  In particular, it is important to investigate the case of contraction semigroups, where, in the Hilbert space case, boundedness of the functional calculus is guaranteed by von Neumann's inequality. However, at this point it is unclear how (and if at all) the latter property is useful when studying Question \ref{Q0}. This is subject to future work.

Finally, we want to study the consequences of the derived results for nonlinear (parabolic) pde's; starting with relatively simple semi-linear equations. Although the notion of admissibility is often perceived as restricted to linear problems, the relation to ISS, a nonlinear concept, seems promising to achieve this step. {\color{black} On the other hand, as mentioned in the introduction, there are recent results by Mironchenko and Wirth \cite{MiWi17} for a class of (nonlinear) systems which the authors call {\it forward complete dynamical (control) systems}, where in particular --- and crucially --- trajectories are required to be continuous. Our results, e.g.\ Theorem \ref{thm11}, provide characterizations of this latter property which moreover implies that the considered parabolic systems then indeed fall into that class. This probably opens the way to use ISS features, such as Lyapunov functions, which becomes particularly interesting in the transition to above mentioned nonlinear parabolic problems, see e.g.\ \cite{ZhengISS} where a semilinear 1-D heat equation is studied.}

\section*{Acknowledgements}
The authors are very grateful to Jonathan R.~Partington for fruitful discussions on the content of the article. They would like to thank Bernhard Haak for helpful comments on the Weiss conjecture for $L^p$. The first and the second named author received funding from the {\it Deutsche Forschungsgemeinschaft (DFG)}, grant number JA 735/12-1. The second named author is further supported by the DFG, grant number RE 2917/4-1. The support is greatly acknowledged. 
{\color{black} Finally the authors are very thankful to the anonymous referees for their careful reading and helpful suggestions.}
\appendix
\section{A technical lemma}
{\color{black}
\begin{lemma}\label{lem:app1}
Let $f\in L^{1}(0,\infty)$ and $f(\tau+\cdot)\in L^{\infty}(0,\infty)$ for some $\tau>0$. Then there exists a Young function $\Phi$ such that $\Phi(|f(\cdot)|)$ is integrable.  Thus, $f\in L_{\Phi}(0,\infty)$. \end{lemma}
\begin{proof}
By 
\cite[Thm.~3.2.5]{Kufner} 
there exists a Young function $\Phi$ such that 
the assertion of the lemma holds in the case that $f$ has bounded essential support (this only relies on the assumption that $f\in L^{1}(0,\infty)$). For the general case consider
$$\int_{0}^{\infty}\Phi(|f(s)|)\,ds=\int_{0}^{\tau}\Phi(|f(s)|)\,ds+\int_{\tau}^{\infty}\frac{\Phi(|f(s)|)}{|f(s)|}|f(s)|\,ds.$$
The first term on the right-hand-side is finite by the first case. For the second term, note that $\frac{\Phi(x)}{x}$ is bounded on any bounded subinterval of $[0,\infty)$, which follows from continuity of $\Phi$ and $\lim_{x\to0^{+}}\frac{\Phi(x)}{x}=0$. Hence, as $f\in L^{\infty}(\tau,\infty)$ and $L^{1}(0,\infty)$, the claim follows.
\end{proof}
}

\def\cprime{$'$}



\end{document}